\documentclass[twoside]{amsart}
\usepackage{amsfonts,amsthm}
\usepackage{enumerate}
\usepackage{amsmath}
\usepackage{amssymb}
\usepackage{amsxtra}
\usepackage{latexsym}

\theoremstyle{plain}

\newtheorem*{thm A}{Theorem~A}
\newtheorem*{thm B}{Theorem~B}
\newtheorem*{thm C}{Theorem~C}
\newtheorem*{main1}{Theorem~1}
\newtheorem*{main2}{Theorem~2}

\newtheorem{theorem}{Theorem}[section]
\newtheorem{lemma}[theorem]{Lemma}

\newtheorem*{pro A}{Proposition~A}
\newtheorem*{pro B}{Proposition~B}

\def \al{\alpha}

\def \SN{\sum_{\nu=1}^3}
\def \D{\mathfrak D}
\def \QP{{\mathcal Q}^{\bot}}
\def \Q{\mathcal Q}
\def \U{\mathfrak U}
\def \CP{{\mathcal C}^{\bot}}
\def \C{\mathcal C}

\def \RX{R_{\xi}}
\def \RXP{R_{\xi}\phi}

\def \X{X_{0}}
\def \x{\xi}
\def \xo{{\xi}_1}
\def \xt{\xi_{2}}
\def \xh{\xi_{3}}
\def \xtw{{\xi}_2}
\def \xth{{\xi}_3}
\def \xn{\xi_{\nu}}
\def \XN{\xi_{\nu}}

\def \ho{\mathcal H}
\def \tu{\mathcal T}

\def \EN{{\eta}_{\nu}}
\def \ENK{{\eta}_{\nu}({\xi})}
\def \eo{\eta_1}
\def \et{\eta_{2}}
\def \eh{\eta_{3}}
\def \etw{\eta_{2}}
\def \e{\eta}

\def \PN{\phi_{\nu}}

\def \PAX{\phi AX}

\def \PNK{{\phi}_{\nu}{\xi}}
\def \PNP{{\phi}_{\nu}{\phi}}
\def \p{\phi}
\def \po{\phi_{1}}
\def \pn{\phi_{\nu}}

\def \Ph{\phi}

\def \PNK{{\phi}_{\nu}{\xi}}

\def \PNP{{\phi}_{\nu}{\phi}}

\def \PAX{{\phi}AX}
\def \PAY{{\phi}AY}

\def \PKN{{\phi}{\xi}_{\nu}}

\def \PNoK{{\phi}_{{\nu}+1}{\xi}}
\def \PNtK{{\phi}_{{\nu}+2}{\xi}}

\def \qnt{q_{{\nu}+2}}
\def \qno{q_{{\nu}+1}}

\def \span{\text{Span}}

\def \N{\nabla}
\def \Na{\nabla}

\def \be{\beta}

\def \la{\lambda}

\def \si{\sigma}
\def \ka{\kappa}

\def \NBo{SU_{2,m-1}/S(U_2{\cdot}U_{m-1})}
\def \NBt{SU_{2,m}/S(U_{2}{\cdot}U_{m})}

\def \HBt{G^{*}_2({\mathbb C}^{m+2})}
\def \GBt{G_2({\mathbb C}^{m+2})}

\def \CC{{\Bbb C}}
\def \RR{{\Bbb R}}

\def \EN{{\eta}_{\nu}}

\def \ENK{{\eta}_{\nu}({\xi})}

\def \ENoK{{\eta}_{{\nu}+1}({\xi})}
\def \ENtK{{\eta}_{{\nu}+2}({\xi})}

\def \KN{{\xi}_{\nu}}

\begin{document}
\title[Commuting Structure Jacobi Operators]{Real hypersurfaces in complex hyperbolic two-plane Grassmannians with commuting structure Jacobi operators}

\vspace{0.2in}

\author[Hyunjin Lee, Young Jin Suh and Changhwa Woo]{Hyunjin Lee, Young Jin Suh and Changhwa Woo}
\address{\newline
Hyunjin Lee
\newline Research Institute of Real and Complex Manifold,
\newline Kyungpook National University,
\newline Daegu 702-701, REPUBLIC OF KOREA}
\email{lhjibis@hanmail.net}

\address{\newline
Young Jin Suh
\newline Department of Mathematics
\newline and Research Institute of Real and Complex Manifold,
\newline Kyungpook National University,
\newline Daegu 702-701, REPUBLIC OF KOREA}
\email{yjsuh@knu.ac.kr}

\address{\newline
Changhwa Woo
\newline Department of Mathematics,
\newline Kyungpook National University,
\newline Daegu 702-701, REPUBLIC OF KOREA}
\email{legalgwch@naver.com}


\thanks{*This work was supported by Grant Proj. No. NRF-2015-R1A2A1A-01002459 and the third author is supported by NRF Grant funded by the Korean Government (NRF-2013-Fostering Core Leaders of Future Basic Science Program).}

%

\begin{abstract}
In this paper, we introduce a new commuting condition between the structure Jacobi operator and symmetric (1,1)-type tensor field $T$, that is, $R_{\xi}\phi T=TR_{\xi}\phi$, where $T=A$ or $T=S$ for Hopf hypersurfaces in complex hyperbolic two-plane Grassmannians. By using simultaneous diagonalzation for commuting symmetric operators, we give a complete classification of real hypersurfaces in complex hyperbolic two-plane Grassmannians with commuting condition respectively.
\end{abstract}
\maketitle

\section*{Introduction}
\setcounter{equation}{0}
\renewcommand{\theequation}{0.\arabic{equation}}
\vspace{0.13in}

It is one of the main topics in submanifold geometry to investigate immersed real hypersurfaces of homogeneous type in Hermitian symmetric spaces of rank~$2$ (HSS2) with certain geometric conditions.
Understanding and classifying real hypersurfaces in HSS2 is one of important problems in differential geometry. One of these spaces is the complex two-plane Grassmannian $\GBt=SU_{2+m}/S(U_2{\cdot}U_m)$ defined by the set of all complex two-dimensional linear subspaces in ${\mathbb C}^{m+2}$.
Another one is the complex hyperbolic two-plane Grassmannian $\HBt=\NBt$ defined by the set of all complex two-dimensional linear subspaces in indefinite complex Euclidean space ${\Bbb C}_2^{m+2}$.

These are typical examples of HSS2. Characterizing typical model spaces of real hypersurfaces under certain geometric conditions is one of our main interests
in the classification theory in $\GBt$ or $\NBt$ (see \cite{S7} and \cite{S8}).

Our recent interest is the study by applying geometric conditions used in submanifolds in $\GBt$ to submanifolds in $\NBt$.

$\GBt=SU_{2+m}/S(U_2{\cdot}U_m)$ has compact transitive group $SU_{2+m}$, however $\NBt$ has noncompact indefinite transitive group $SU_{2,m}$.
This distinction gives various remarkable results.

The complex hyperbolic two-plane Grassmannian $SU_{2,m}/S(U_2{\cdot}U_m)$ is the unique noncompact, irreducible, K\"{a}hler and quaternionic K\"{a}hler manifold
which is not a hyperk\"{a}hler manifold.

\vskip 3pt
Let $M$ be a real hypersurface in complex hyperbolic two-plane Grassmannian $\NBt$.
Let $N$ be a local unit normal vector field on $M$. Since the complex hyperbolic two-plane Grassmannians $\NBt$ has the K\"{a}hler structure $J$, we may define a {\it Reeb vector field} $\xi =-JN$ and a $1$-dimensional distribution $\CP=\text{Span}\{\,\xi\}$.
\vskip 3pt
Let $\C$ be the orthogonal complement of distribution $\CP$ in $T_{p}M$ at $p\in M$. It is the complex maximal subbundle of $T_{p}M$.
Thus the tangent space of $M$ consists of the direct sum of $\C$ and $\CP$ as follows: $T_{p}M =\C\oplus \CP$.
The real hypersurface $M$ is said to be {\it Hopf} if $A\C\subset\C$, or equivalently,
the Reeb vector field $\x$ is principal with principal curvature $\al=g(A\x,\x)$, where $g$ denotes the metric. In this case,
the principal curvature $\al$ is said to be a {\it Reeb curvature} of~$M$.
\vskip 3pt
From the quaternionic K\"{a}hler structure $\mathfrak J=\span\{J_1, J_2, J_3\}$ of $\NBt$, there naturally exist {\it almost contact 3-structure} vector fields $\xi_{\nu}=-J_{\nu}N$, $\nu=1,2,3$.
Let $\QP = \text{Span}\{\,\xi_1, \xi_2, \xi_3\}$. It is a 3-dimensional distribution in the tangent space $T_{p}M$ of $M$ at $p \in M$. In addition, $\Q$ stands for the orthogonal complement of $\QP$ in $T_{p}M$.
It is the quaternionic maximal subbundle of $T_{p}M$. Thus the tangent space of $M$ can be splitted into $\Q$ and $\QP$ as follows: $T_{p}M =\Q\oplus \QP$.
\vskip 3pt
Thus, we have considered two natural geometric conditions for real hypersurfaces in
$SU_{2,m}/S(U_2{\cdot}U_m)$ such that the subbundles $\mathcal C$ and $\Q$ of $TM$ are both invariant
under the shape operator.  By using these geometric conditions, we will use the results in Suh \cite[Theorem 1]{S7}.
\vskip 6pt
\par

On the other hand, a Jacobi field along geodesics of a given Riemannian manifold $(\bar M, \bar g)$ plays an important role in the study of differential geometry. It satisfies a well-known differential equation which inspires Jacobi operators. It is defined by $(\bar R_{X}(Y))(p)= (\bar R(Y, X)X)(p)$, where $\bar R$ denotes the curvature tensor of~$\bar M$ and $X$,~$Y$ denote any vector fields on $\bar M$. It is known to be a self-adjoint endomorphism on the tangent space $T_{p}\bar M$, $p \in \bar M$. Clearly, each tangent vector field $X$ to $\bar M$ provides a Jacobi operator with respect to $X$. Thus the Jacobi operator on a real hypersurface $M$ of $\bar M$ with respect to $\x$ is said to be a {\it structure Jacobi operator} and will be denoted by $R_{\xi}$. The Riemannian curvature tensor of $M$ (resp., $\bar M$) is denoted by $R$ (resp., $\bar R$).

\vskip 3pt
%

For a commuting problem concerned with the structure Jacobi operator $R_{\xi}$ and the structure tensor $\p$ of Hopf hypersurface $M$ in $\GBt$, that is, $\RXP A= A\RXP$, Lee, Suh and Woo \cite{LSW} proved that
a Hopf hypersurface $M$ with $\RXP A= A\RXP$ and $\x\al=0$ is locally congruent to an open part of a tube around a totally geodesic
$G_2({\mathbb C}^{m+1})$ in $G_2({\mathbb C}^{m+2})$. Motivated by this result, we consider the same condition in the different ambient space, that is,
\begin{equation}\label{C-1}
R_{\xi}\p  AX =A R_{\xi}\p X
\tag{C-1}
\end{equation}
for any tangent vector field $X$ on $M$ in $\NBt$. The geometric meaning of $R_{\xi}\p  AX =A R_{\xi}\p X$ can be explained in such a way that any eigenspace of $R_{\x}$ on the distribution ${\C}=\{X \in T_{p}M \mid X\perp \xi\}$, $p \in M$,
is invariant under the shape operator $A$ of $M$ in $\NBt$. Then by using \cite[Theorem 1]{S7}, we give a complete classification of Hopf hypersurfaces in $\NBt$ with $R_{\xi}\p  AX =A R_{\xi}\p X$ as follows:
\vskip 3pt
\begin{main1}
Let $M$ be a Hopf hypersurface in complex hyperbolic two-plane Grassmannians $\NBt$, $m \geq 3$ with  $\RXP  A=A \RXP$.
If the Reeb curvature $\al=g(A\x, \x)$ is constant along the Reeb direction of the structure vector field $\x$, then $M$ is locally congruent to one of the following:
\begin{enumerate}[\rm (i)]
\item {a tube over a totally geodesic $\NBo$ in $\NBt$ or}
\item {a horosphere in $\NBt$ whose center at infinity is singular and of type $JX \in {\mathfrak J}X$.}
\end{enumerate}
\end{main1}

From the Riemannian curvature tensor $R$ of $M$ in $\NBt$ we can define the Ricci tensor $S$ of $M$ in such a way that
$$g(SX,Y)={\sum}_{i=1}^{4m-1}g(R(e_i,X)Y,e_i),$$ where $\{e_1, {\cdots}, e_{4m-1}\}$
denotes a basis of the tangent space $T_pM$ of $M$, $p{\in}M$, in $\NBt$ (see~\cite{SW}).
Then we can consider another new commuting condition
\begin{equation}\label{C-2}
\RXP SX = S \RXP X
\tag{C-2}
\end{equation}
for any tangent vector field $X$ on $M$. That is, the operator $\RXP$ commutes with the Ricci tensor $S$.

Then by \cite[Theorem 1]{S7}, we also give another classification related to the Ricci tensor $S$ of $M$ in $\NBt$ as follows:

\begin{main2}
Let $M$ be a Hopf hypersurface in complex hyperbolic two-plane Grassmannians $\NBt$, $m \geq 3$ with $\RXP  S=S \RXP$.
If the smooth function $\al=g(A\x, \x)$ is constant along the direction of $\x$, then $M$ is locally congruent to one of the following:
\begin{enumerate}[\rm (i)]
\item {a tube over a totally geodesic $\NBo$ in $\NBt$ or}
\item {a horosphere in $\NBt$ whose center at infinity is singular and of type $JX \in {\mathfrak J}X$.}
\end{enumerate}
\end{main2}

\vskip 3pt

\vskip 3pt

In this paper, we refer \cite{PeSW}, \cite{S7}, \cite{S8} and \cite{SW} for Riemannian geometric structures of complex hyperboilc two-plane Grassmannians $\NBt$, $m \geq 3$.
\vskip 15pt
%
\section{The complex hyperbolic two-plane Grassmannian $\NBt$}\label{section 1}
\setcounter{equation}{0}
\renewcommand{\theequation}{1.\arabic{equation}}
\vspace{0.13in}

In this section we summarize basic material about complex hyperbolic
two-plane Grassmann manifolds $\NBt$, for details we refer to
 \cite{PeSuWa}, \cite{S5}, \cite{S7} and \cite{SW}. The Riemannian symmetric space
$SU_{2,m}/S(U_2{\cdot}U_m)$, which consists of all complex
two-dimensional linear subspaces in indefinite complex Euclidean
space $\CC_2^{m+2}$ is a connected, simply connected,
irreducible Riemannian symmetric space of noncompact type and with
rank two. Let $G = SU_{2,m}$ and $K = S(U_2{\cdot}U_m)$, and denote
by ${\frak g}$ and ${\frak k}$ the corresponding Lie algebra of the
Lie group $G$ and $K$  respectively. Let $B$ be the Killing form of
${\frak g}$ and denote by ${\frak p}$ the orthogonal complement of
${\frak k}$ in ${\frak g}$ with respect to $B$. The resulting
decomposition ${\frak g} = {\frak k} \oplus {\frak p}$ is a Cartan
decomposition of ${\frak g}$. The Cartan involution $\theta \in
{\text Aut}({\frak g})$ on ${\frak s}{\frak u}_{2,m}$ is given by
$\theta(A) = I_{2,m} A I_{2,m}$, where
\begin{equation*}
\begin{split}
I_{2,m} = \begin{pmatrix}
-I_{2} & 0_{2,m} \\
0_{m,2} & I_{m} \end{pmatrix},
\end{split}
\end{equation*}
$I_2$ and $I_m$ denote the identity $2 \times 2$-matrix and $m
\times m$-matrix respectively. Then $< X , Y > = -B(X,\theta Y)$
becomes a positive definite ${\text Ad}(K)$-invariant inner product
on ${\frak g}$. Its restriction to ${\frak p}$ induces a metric $g$
on $SU_{2,m}/S(U_2{\cdot}U_m)$, which is also known as the Killing
metric on $SU_{2,m}/S(U_2{\cdot}U_m)$. Throughout this paper we
consider $SU_{2,m}/S(U_2{\cdot}U_m)$ together with this particular
Riemannian metric $g$.
\par
\vskip 6pt The Lie algebra ${\frak k}$ decomposes orthogonally into
${\frak k}  = {\frak s}{\frak u}_2 \oplus {\frak s}{\frak u}_m
\oplus {\frak u}_1$, where ${\frak u}_1$ is the one-dimensional
center of ${\frak k}$. The adjoint action of ${\frak s}{\frak u}_2$
on ${\frak p}$ induces the quaternionic K\"{a}hler structure ${\frak
J}$ on $SU_{2,m}/S(U_2{\cdot}U_m)$, and the adjoint action of
\begin{equation*}
\begin{split}
Z = \begin{pmatrix}
 \frac{mi}{m+2}I_2 & 0_{2,m} \\
 0_{m,2} & \frac{-2i}{m+2}I_m
 \end{pmatrix}
 \in {\frak u}_1
\end{split}
\end{equation*}
induces the K\"{a}hler structure $J$ on $SU_{2,m}/S(U_2{\cdot}U_m)$.
By construction, $J$ commutes with each almost Hermitian structure
$J_{\nu}$ in ${\frak J}$ for ${\nu}=1,2,3$. Recall that a canonical
local basis $\{J_1,J_2,J_3\}$ of a quaternionic K\"{a}hler structure
${\frak J}$ consists of three almost Hermitian structures
$J_1,J_2,J_3$ in ${\frak J}$ such that $J_{\nu} J_{{\nu}+1} = J_{{\nu} +
2} = - J_{{\nu}+1} J_{\nu}$, where the index ${\nu}$ is to be taken modulo
$3$. The tensor field $JJ_{\nu}$, which is locally defined on
$SU_{2,m}/S(U_2{\cdot}U_m)$, is self-adjoint and satisfies
$(JJ_{\nu})^2 = I$ and ${\text tr}(JJ_{\nu}) = 0$, where $I$ is the
identity transformation. For a nonzero tangent vector $X$, we define
${\Bbb R}X = \{\lambda X \vert \lambda \in {\Bbb R}\}$, ${\Bbb C}X =
{\Bbb R}X \oplus {\Bbb R}JX$, and ${\Bbb H}X = {\Bbb R}X \oplus
{\frak J}X$.
\par
\vskip 6pt
We identify the tangent space $T_oSU_{2,m}/S(U_2{\cdot}U_m)$ of
$SU_{2,m}/S(U_2{\cdot}U_m)$ at $o$ with ${\frak p}$ in the usual way. Let
${\frak a}$ be a maximal abelian subspace of ${\frak p}$. Since
$SU_{2,m}/S(U_2{\cdot}U_m)$ has rank two, the dimension of any such
subspace is two. Every nonzero tangent vector $X \in
T_oSU_{2,m}/S(U_2{\cdot}U_m) \cong {\frak p}$ is contained in some maximal
abelian subspace of ${\frak p}$. Generically this subspace is
uniquely determined by $X$, in which case $X$ is called regular. If
there exist more than one maximal abelian subspaces of ${\frak p}$
containing $X$, then $X$ is called singular. There is a simple and
useful characterization of the singular tangent vectors: A nonzero
tangent vector $X \in {\frak p}$ is singular if and only if $JX \in
{\frak J}X$ or $JX \perp {\frak J}X$.
\par
\vskip 6pt
 Up to scaling there exists a unique $SU_{2,m}$-invariant Riemannian metric $g$ on $\NBt$. Equipped with this
metric, $\NBt$ is a Riemannian symmetric space of rank two which is
both K\"ahler and quaternionic K\"ahler. For computational reasons
we normalize $g$ such that the minimal sectional curvature of
$({\NBt},g)$ is $-4$. The sectional curvature $K$ of the noncompact
symmetric space $SU_{2,m}/S(U_2{\cdot}U_m)$ equipped with the
Killing metric $g$ is bounded by $-4{\leq}K{\leq}0$. The sectional
curvature $-4$ is obtained for all two-planes ${\Bbb C}X$ when $X$
is a non-zero vector with $JX \in {\frak J}X$.
\par
\vskip 6pt

When $m=1$,  $G_2^{*}(\CC^3)=SU_{1,2}/S(U_1{\cdot}U_2)$ is isometric
to the two-dimensional complex hyperbolic space $\CC H^2$ with
constant holomorphic sectional curvature $-4$.
\par
When $m=2$, we note that the isomorphism $SO(4,2)\simeq SU_{2,2}$ yields an isometry between $G_2^{*}(\CC ^4)=SU_{2,2}/S(U_2{\cdot}U_2)$ and
the indefinite real Grassmann manifold $G_2^{*}(\RR_2^6)$ of oriented
two-dimensional linear subspaces of an indefinite Euclidean space $\RR_2^6$.
For this reason we
assume $m \geq 3$ from now on, although many of the subsequent
results also hold for $m = 1,2$.
\par
\vskip 6pt

From now on, hereafter $X$,$Y$ and $Z$ always stand for any tangent vector fields on $M$.

The Riemannian curvature tensor $\bar{R}$ of $\NBt$ is locally given
by
\begin{equation*}
\begin{split}
-2\bar{R}(X,Y)Z = &g(Y,Z)X - g(X,Z)Y  +  g(JY,Z)JX\\
&\qquad - g(JX,Z)JY- 2g(JX,Y)JZ \\
&\qquad  + \SN \{g(J_{\nu} Y,Z)J_{\nu} X - g(J_{\nu} X,Z)J_{\nu}Y- 2g(J_{\nu} X,Y)J_{\nu} Z\}\\
&\qquad  + \SN \{g(J_{\nu} JY,Z)J_{\nu} JX - g(J_{\nu} JX,Z)J_{\nu}JY\},
\end{split}
\end{equation*}
where $\{J_1,J_2,J_3\}$ is any canonical local basis of ${\frak J}$.

\par
\vskip 20pt

\section{Fundamental formulas in $\NBt$}
\setcounter{equation}{0}
\renewcommand{\theequation}{2.\arabic{equation}}
\vspace{0.13in}

In this section, we derive some basic formulas and the Codazzi
equation for a real hypersurface in $\NBt$ (see
 \cite{S7}, \cite{S8} and \cite{SW}).
\par
\vskip 6pt
Let $M$ be a real hypersurface in complex hyperbolic two-plane Grassmannian $\NBt$, that is, a hypersurface in
$\NBt$ with real codimension one. The induced Riemannian metric on
$M$ will also be denoted by $g$, and $\nabla$ denotes the Levi
Civita covariant derivative of $(M,g)$. We denote by $\mathcal C$
and $\mathcal Q$ the maximal complex and quaternionic subbundle of
the tangent bundle $TM$ of $M$, respectively. Now let us put
\begin{equation}\label{eq: 2.1}
JX={\phi}X+{\eta}(X)N,\quad J_{\nu}X={\phi}_{\nu}X+{\eta}_{\nu}(X)N
\end{equation}
for any tangent vector field $X$ of a real hypersurface $M$ in
$\NBt$, where ${\phi}X$ denotes the tangential component of $JX$ and
$N$ a unit normal vector field of $M$ in $\NBt$.
\par
\vskip 6pt
From the K\"{a}hler structure $J$ of $\NBt$ there exists
an almost contact metric structure $(\phi,\xi,\eta,g)$ induced on
$M$ in such a way that
\begin{equation}\label{eq: 2.2}
\phi^{2}X=-X+\eta(X)\xi,\quad \eta(\xi)=1,\quad \phi \xi =0, \quad \eta(X)=g(X,\xi)
\end{equation}
for any vector field $X$ on $M$. Furthermore, let $\{J_1,J_2,J_3\}$ be a canonical local basis of ${\mathfrak J}$. Then the quaternionic K\"{a}hler structure $J_{\nu}$ of $\NBt$, together with the condition $J_{\nu}J_{{\nu}+1} = J_{{\nu}+2} = -J_{{\nu}+1}J_{\nu}$ in section~$1$, induces an almost contact metric 3-structure $(\phi_{\nu}, \xi_{\nu}, \eta_{\nu}, g)$ on $M$ as follows:
\begin{equation}\label{eq: 2.3}
\begin{split}
& \phi_{\nu}^{2}X=-X+\eta_{\nu}(X)\xi_{\nu},\quad \eta_{\nu}(\xi_{\nu})=1, \quad  \phi_{\nu} \xi_{\nu} =0, \\
&{\phi}_{{\nu} +1}{\xi}_{\nu}=-{\xi}_{{\nu}+2},\quad {\phi}_{\nu}{\xi}_{{\nu}+1}={\xi}_{{\nu}+2},\\
&{\phi}_{\nu}{\phi}_{{\nu}+1}X
= {\phi}_{{\nu}+2}X+{\eta}_{{\nu}+1}(X){\xi}_{\nu},\\
&{\phi}_{{\nu}+1}{\phi}_{\nu}X=-{\phi}_{{\nu}+2}X+{\eta}_{\nu}(X)
{\xi}_{{\nu}+1}
\end{split}
\end{equation}
for any vector field $X$ tangent to $M$. Moreover, from the commuting property of $J_{\nu}J=JJ_{\nu}$, ${\nu}=1,2,3$ in section~\ref{section 1} and (\ref{eq: 2.1}), the relation between these two contact metric structures $(\phi, \xi, \eta, g)$ and $(\phi_{\nu}, \xi_{\nu}, \eta_{\nu}, g)$, $\nu=1,2,3$, can be given by
\begin{equation}\label{eq: 2.4}
\begin{split}
&{\Ph}{\Ph}_{\nu}X  ={\Ph}_{\nu}{\Ph}X+{\eta}_{\nu}(X){\xi}-{\eta}(X)
{\xi}_{\nu},\\
&\eta_{\nu}(\phi X) = \eta(\phi_{\nu}X), \quad {\phi}{\xi}_{\nu}={\phi}_{\nu}{\xi}.
\end{split}
\end{equation}

\def \E{\eta}

\noindent On the other hand, from the parallelism of K\"{a}hler structure $J$,~that is,~$\widetilde \nabla J=0$ and the quaternionic K\"{a}hler structure $\mathfrak J$, together with Gauss and Weingarten formulas, it follows that
\begin{equation}\label{eq: 2.5}
({\N}_X{\Ph})Y={\E}(Y)AX-g(AX,Y){\xi},\quad {\nabla}_X{\xi}={\Ph}AX,
\end{equation}
\begin{equation}\label{eq: 2.6}
{\N}_X{\xi}_{\nu}=q_{{\nu}+2}(X){\xi}_{{\nu}+1}-q_{{\nu}+1}(X)
{\xi}_{{\nu}+2}+{\Ph}_{\nu}AX,
\end{equation}
\begin{equation}\label{eq: 2.7}
\begin{split}
({\N}_X{\Ph}_{\nu})Y&=-q_{{\nu}+1}(X){\Ph}_{{\nu}+2}Y+q_{{\nu}+2}(X)
{\Ph}_{{\nu}+1}Y +{\E}_{\nu}(Y)AX\\
&\quad \ \  -g(AX,Y){\xi}_{\nu}.
\end{split}
\end{equation}
\par
\noindent Combining these formulas, we find the following:
\begin{equation} \label{eq: 2.8}
\begin{split}
{\N}_X({\Ph}_{\nu}{\xi})&={\N}_X({\Ph}{\xi}_{\nu})\\
&=({\N}_X{\Ph}){\xi}_{\nu}+{\Ph}({\N}_X{\xi}_{\nu})\\
&=q_{{\nu}+2}(X){\Ph}_{{\nu}+1}{\xi}-q_{{\nu}+1}(X){\Ph}_{{\nu}+2}
{\xi}+{\Ph}_{\nu}{\Ph}AX\\
&\ \ \ -g(AX,{\xi}){\xi}_{\nu}+{\E}({\xi}_{\nu})AX.
\end{split}
\end{equation}

 Finally, using the explicit expression for the Riemannian curvature tensor $\bar{R}$ of
$\NBt$ in \cite{S8}, the Codazzi equation takes the form
\begin{equation}\label{eq: 2.9}
\begin{split}
-2(\nabla_XA)Y &+2 (\nabla_YA)X
 =\eta(X)\phi Y - \eta(Y)\phi X - 2g(\phi X,Y)\xi \\
& \qquad + \SN \big\{\eta_{\nu}(X)\phi_{\nu} Y - \eta_{\nu}(Y)\phi_{\nu}
 X - 2g(\phi_{\nu} X,Y)\xi_{\nu}\big\} \\
& \qquad + \SN \big\{\eta_{\nu}(\phi X)\phi_{\nu}\phi Y
- \eta_{\nu}(\phi Y)\phi_{\nu}\phi X\big\} \\
& \qquad + \SN \big\{\eta(X)\eta_{\nu}(\phi Y)
- \eta(Y)\eta_{\nu}(\phi X)\big\}\xi_{\nu},
\end{split}
\end{equation}
for any vector fields $X$ and $Y$ on $M$.

On the other hand, by differentiating $A{\xi}={\alpha}{\xi}$ and using \eqref{eq: 2.9}, we get
the following
\begin{equation}\label{eq: 2.10}
\begin{split}
g(&{\phi}X,Y)-{\SN}\{{\EN}(X){\EN}({\phi}Y)-{\EN}(Y){\EN}({\phi}X)-g({\PN}X,Y){\EN}({\xi})\}\\
&=g(({\Na}_XA)Y-({\Na}_YA)X,{\xi})\\
&=g(({\Na}_XA){\xi},Y)-g(({\Na}_YA){\xi},X)\\
&=(X{\alpha}){\eta}(Y)-(Y{\alpha}){\eta}(X)+{\alpha}g((A{\phi}+{\phi}A)X,Y)-2g(A{\phi}AX,Y).\\
\end{split}
\end{equation}
Putting $X={\xi}$ gives
\begin{equation}\label{eq: 2.10-1}
Y{\alpha}=({\xi}{\alpha}){\eta}(Y)+2{\SN}{\EN}({\xi}){\EN}({\phi}Y).
\end{equation}
Then, substituting \eqref{eq: 2.10-1} into \eqref{eq: 2.10}
the above equation, we have the following
\begin{equation}\label{eq: 2.11}
\begin{split}
A{\phi}AY&=\frac{\alpha}{2}(A{\phi}+{\phi}A)Y
+{\SN}\big\{{\eta}(Y){\EN}({\xi}){\phi}{\KN}+{\EN}({\xi}){\EN}({\phi}Y){\xi}\big\}\\
&\quad-\frac{1}{2}{\phi}Y-\frac{1}{2}{\SN}\big\{{\EN}(Y){\phi}{\KN}+{\EN}({\phi}Y){\KN}+{\EN}({\xi}){\PN}Y\big\}.
\end{split}
\end{equation}

\noindent By differentiating and using (\ref{eq: 2.4}), (\ref{eq: 2.5}) and (\ref{eq: 2.6}), we have
\begin{equation*}
\begin{split}
{\nabla}_X({\rm grad}\ {\alpha})& = X({\xi}{\alpha}){\xi}+({\xi}{\alpha}){\phi}AX\\
&\quad -2{\SN}\Big\{{\qnt}(X){\ENoK}-{\qno}(X){\ENtK}+2{\EN}({\PAX})\Big\}{\PKN}\\
&\quad -2{\SN}{\ENK}\Big\{-{\qno}(X){\PNtK}+{\qnt}(X){\PNoK}+{\ENK}AX\\
&\quad\quad\quad\quad\quad\quad\quad -g(AX,{\xi}){\KN}+{\PNP}AX\Big\}\\
&=X({\xi}{\alpha}){\xi}+({\xi}{\alpha}){\phi}AX-4{\SN}\EN(\PAX){\PKN}\\
&\quad -2{\SN}{\ENK}\Big\{{\ENK}AX-g(AX,{\xi}){\KN}+{\PNP}AX \Big\}.
\end{split}
\end{equation*}

\noindent By taking the skew-symmetric part to the above equation, we have
\begin{equation*}
\begin{split}
0&=X({\xi}{\alpha}){\eta}(Y)-Y({\xi}{\alpha}){\eta}(X)+({\xi}{\alpha})g\big((A{\phi}+{\phi}A)X,Y\big)\\
&\quad -4{\SN}\Big\{{\EN}({\phi}AX)g({\PKN},Y)-{\EN}({\PAY})g({\PKN},X)\Big\}\\
&\quad +2{\alpha}{\SN}{\ENK}\Big\{{\eta}(X){\EN}(Y)-{\eta}(Y){\EN}(X)\Big\}\\
&\quad -2{\SN}{\ENK}\Big\{g({\PNP}AX,Y)-g({\PNP}AY,X)\Big\}.
\end{split}
\end{equation*}

\noindent From this, by putting $X={\xi}$ we have the following
\begin{equation}\label{eq: 2.12}
Y({\xi}{\alpha})={\xi}({\xi}{\alpha}){\eta}(Y)+2{\alpha}{\SN}{\ENK}{\EN}(Y)-2{\SN}{\ENK}{\EN}(AY).
\end{equation}

\noindent From this, if we assume that ${\xi}{\alpha}=0$, then it follows that
\begin{equation*}
{\SN}{\ENK}{\EN}(AX)={\alpha}{\SN}{\ENK}{\EN}(X).
\end{equation*}

\par
\vskip 6pt

\begin{lemma}\label{lemma 2.2}
Let $M$ be a Hopf real hypersurface in $\NBt$. If the principal curvature $\alpha$ is constant
along the direction of $\xi$, then the distribution $\Q$ or $\QP$ component of the structure vector field
$\xi$ is invariant by the shape operator.
\end{lemma}

%

\section{Proof of Theorem 1}\label{section 3}
\setcounter{equation}{0}
\renewcommand{\theequation}{3.\arabic{equation}}

\vspace{0.13in}

Let $M$ be a Hopf hypersurface in $\NBt$ with
\begin{equation}
R_{\xi}\p AX = A R_{\xi}\p X.
\tag{C-1}
\end{equation}
The structure Jacobi operator $R_{\xi}$ of $M$ is defined by $R_{\xi}X=R(X,\xi)\xi$ for any tangent vector $X \in T_{p}M$, $p \in M$ (see \cite{JPS} and \cite{PMS}).
Then for any tangent vector field $X$ on $M$ in $\NBt$, we calculate the structure Jacobi operator $R_{\xi}$
\begin{equation}\label{eq: 3.2}
\begin{split}
2R_{\xi}(X)&=2R(X,{\xi}){\xi}\\
&= -X + {\eta}(X){\xi} + {\SN}\big\{{\EN}(X){\KN}-{\eta}(X){\ENK}{\KN}\\
&\ \ +3\EN(\p X){\PNK}+{\ENK}{\PNP}X\big\}+2{\alpha}AX-2{\eta}(AX)A{\xi},
\end{split}
\end{equation}
where $\alpha$ denotes the Reeb curvature defined by $g(A{\xi},{\xi})$.


\vskip 3pt

\begin{lemma}\label{lemma 3.1}
Let $M$ be a Hopf hypersurface in $\NBt$ with the commuting condition $R_{\xi}\p AX = A R_{\xi}\p X$.
If the smooth function $\alpha$ is constant along the direction of $\x$ on $M$,
then the Reeb vector field $\x$ belongs to either the distribution $\Q$ or the distribution~$\QP$.
\end{lemma}

\begin{proof}
To prove this lemma, without loss of generality, $\x$ may be written as
\begin{equation}\label{*}
\xi = \eta(\X)\X + \e(\xo)\xo
\tag{*}
\end{equation}
where $X_{0}$ (resp., $\xo$) is a unit vector in $\Q$ (resp., $\QP$) and $\e(\X)\e(\xo) \neq 0$.
\vskip 3pt
From \eqref{*} and $\p\x=0$, we have
\begin{equation}\label{eq: 3.3}
\left \{
\begin{aligned}
&\p\X=-\e(\xo)\po\X,\\
&\p\xo=\po\x=\e(\X)\po\X,\\
&\po\p\X=\eo(\x)\X.
\end{aligned}
\right.
\end{equation}

Let $\mathfrak U=\{p \in M\,|\, \alpha(p) \neq 0 \}$ be an open subset of $M$. From now on, we discuss our arguments on $\U$.
\noindent By virtue of Lemma~\ref{lemma 2.2}, $\x\al=0$ gives $A\X=\alpha \X$ and $A\xo=\al\xo$. The equation \eqref{eq: 2.11} yields $\alpha A\p X_0  = ({\alpha}^2-2{\eta}^2(X_0))\p X_0$
by substituting $X=X_0$.
Since $\alpha $ is non-vanishing on $\U$, it becomes
\begin{equation}\label{eq: 3.4}
A\p X_0  = \sigma\p X_0,
\end{equation}
where $\sigma=\frac{{\alpha}^2-2{\eta}^2(X_0)}{\alpha}$.

\noindent From \eqref{eq: 3.3} and \eqref{eq: 3.4}, we have
\begin{equation}\label{eq: 3.5}
\left \{
\begin{aligned}
&\RX(\X)=\al^2\X-\al^2\e(\X)\x,\\
&\RX(\xo)=\al^2\xo-\al^2\e(\xo)\x,\\
&\RX(\p\X)=\big(\al^2-4\e^2(\X)\big)\p\X.
\end{aligned}
\right.
\end{equation}

%
On $\mathfrak U$, substituting $X$ by $\p X_0$ into \eqref{C-1}, we have
\begin{equation}\label{eq: 3.6}
\X-\e(\X)\x=0,
\end{equation}

\noindent which is a contradiction. Therefore, $\U=\emptyset$, and thus it must be $p \in M-\U$. Since the set $M - \U=\text{Int}(M - \U) \cup \partial (M - \U)$, we consider the following two cases.
Here $\text{Int}$ (resp., $\partial$) denotes an interior (resp., the boundary) of $(M-\U)$.
\begin{itemize}
\item {{\bf Case 1.} $p \in \text{Int}(M-\U)$.}
\end{itemize}

\vskip 3pt
If $p \in \text{Int}(M-\U)$, then $\alpha = 0$. For this case, it was proved by the equation~\eqref{eq: 2.10-1}.

\vskip 7pt

\begin{itemize}
\item {{\bf Case 2.} $p \in \partial (M-\U)$.}
\end{itemize}

\vskip 2pt

Since $p \in \partial M-\U$, there exists a sequence of points $p_{n}$ such that $p_{n} \rightarrow p$ with $\al(p)=0$ and $\al(p_{n})\neq 0$. Such a sequence will have an infinite subsequence where $\eta(\xo)=0$ (in which case $\x \in \Q$ at $p$, by the continuity) or an infinite subsequence where $\eta(\X)=0$ (in which case $\x \in \Q^{\bot}$ at $p$).

\vskip 3pt
Accordingly, we get a complete proof of our lemma.
\end{proof}
%
\vskip 3pt
\noindent From Lemma~\ref{lemma 3.1}, we consider the case that $\x$ belongs to the distribution $\QP$. Thus without loss of generality, we may put $\x=\xo$.
\noindent Differentiating $\x=\xi_1$ along any direction $X \in TM$ and using \eqref{eq: 2.5} and \eqref{eq: 2.6}, it gives us
\begin{equation}\label{eq: 3.7}
2\eta_{3}(AX)\xt-2\eta_{2}(AX)\xh+\po AX-\p AX=0.
\end{equation}
\noindent Then, by using the symmetric (resp., skew-symmetric) property of the shape operator $A$ (resp., the structure tensor field $\p$), we also obtain
\begin{equation}\label{eq: 3.8}
2\eh(X)A\xt-2\et(X)A\xh+A\po X-A\p X=0.
\end{equation}

\noindent Applying $\po$ to \eqref{eq: 3.7}, it implies
\begin{equation}\label{eq: 3.9}
\begin{split}
2\eta_{3}(AX)\x_{3}+2\eta_{2}(AX)\x_{2}-AX+\al\e(X)\x-\po\p AX=0.
\end{split}
\end{equation}

\noindent On the other hand, replacing $X=\p X$ into \eqref{eq: 3.7}, we have
\begin{equation}\label{eq: 3.10}
\begin{split}
-2\eta_{2}(X)A\x_{2}-2\eta_{3}(X)A\x_{3}+A\po\p X-AX-\al\e(X)\x=0.
\end{split}
\end{equation}

\begin{lemma}\label{lemma 3.2}
Let $M$ be a Hopf hypersurface in $\NBt$, $m \geq 3$, with $R_{\xi}\p A=AR_{\xi}\p$. If the Reeb vector field $\x$ belongs to the distribution $\QP$,
then the shape operator $A$ commutes with the structure tensor field $\p$.
\end{lemma}

\begin{proof}
Applying $\x=\xo$ into right hand side (resp., left hand side) of \eqref{C-1}, we get
\begin{equation*}
\begin{split}
2R_{\xi}\p AX&=-A\p X+2\alpha A^{2}\p X-2\eta_{3}(X)A\xt+2\eta_{2}(X)A\xh-A\p_{1}X,\\
2AR_{\xi}\p X&=-\p AX+2\alpha A\p A X-2\eta_{3}(AX)\xt+2\eta_{2}(AX)\xh-\phi_{1} AX.
\end{split}
\end{equation*}

\noindent Combining \eqref{eq: 3.7} and \eqref{eq: 3.8}, the above equations become
\begin{equation*}
\begin{split}
R_{\xi}\p AX&=-A\p X+\alpha A^{2}\p X,\\
AR_{\xi}\p X&=-\p AX+\alpha A\p A X.
\end{split}
\end{equation*}

\noindent Hence, \eqref{C-1} is equivalent to
\begin{equation}\label{eq: 3.12}
A\p-\p A=\al A(A\p-\p A)
\end{equation}

\noindent Taking the symmetric part of \eqref{eq: 3.12}, we have
\begin{equation}\label{eq: 3.13}
A\p-\p A=\al (A\p-\p A)A.
\end{equation}

\noindent From this, we can divide into the following three cases:
\vskip 3pt
First, let us consider an open subset $\mathfrak U=\{ p \in M \,|\, \alpha (p) \neq 0\}$ of $M$. Naturally we can apply \eqref{eq: 3.12} and \eqref{eq: 3.13} on the open subset $\U$.
\begin{equation*}
(A\p-\p A)AX=A(A\p-\p A)X.
\end{equation*}
Since the shape operator $A$ and the tensor $A\p-\p A$ are both symmetric operators and commute with each other, there exists a common orthonormal basis $\{E_i\}_{i=1,...,4m-1}$ which gives a simultaneous diagonalization.
Specifically, we have
\begin{eqnarray}
& AE_i=\lambda_i E_i,\label{eq: 3.14} \\
& (A\p-\p A )E_i=\beta_iE_i,\label{eq: 3.15}
\end{eqnarray}
where $\lambda_i$ and $\beta_i$ are scalars for all $i=1,2,...,4m-1$.

\vskip 3pt
Taking the inner product with $E_i$ into \eqref{eq: 3.15}, we have
\begin{equation}\label{eq: 3.16}
\beta_i g(E_i,E_i) =g\big((A\p-\p A)E_i,E_i\big)= 2\lambda_ig(\p E_i,E_i)=0.
\end{equation}
Since $g(E_i,E_i)=1$, $\beta_i=0$ for all $i=1,2,...,4m-1$. Hence $A\p X = \p AX$ for any tangent vector field $X$ on $\mathfrak U$.

\vskip 3pt

Next, if $p \in \text{Int}(M-\U)$, then $\alpha(p)=0$. From this, the equation~\eqref{eq: 3.13} gives $(A\p - \p A)X(p)=0$.

\vskip 3pt

Finally, let us assume that $p \in \partial (M-\U)$, where $\partial (M-\U)$ is the boundary of~$M-\U$. Then there exists a subsequence $\{p_{n}\} \subset \U$ such that $p_{n}\rightarrow p$. Since $(A \phi - \phi A)X(p_{n})=0$ on the open subset~$\U$ in $M$, by the continuity we also get $(A \phi - \phi A)X(p)=0$.

\vskip 3pt

Summing up these observations, it is natural that the shape operator $A$ commutes with the structure tensor field $\p$ under our assumption.
\end{proof}

By \cite{S5} we assert~$M$ with the assumptions given in lemma~\ref{lemma 3.2} is locally congruent to one of the following
hypersurfaces:
\begin{enumerate}
\item [$(\tu_{A})$] {a tube over a totally geodesic $\NBo$ in $\NBt$ or,}
\item [$(\ho_{A})$] {a horosphere in $\NBt$ whose center at infinity is singular and of type $JX \in {\mathfrak J}X$.}
\end{enumerate}

\vskip 3pt
In a paper due to \cite{S5}, Suh gave some information related to the shape operator~$A$ of $\tu_{A}$ and $\ho_{A}$ as follows:
\begin{pro A}\label{Proposition A}
Let $M$ be a connected real hypersurface in complex hyperbolic two-plane Grassamannian $SU_{2,m}/S(U_2U_m)$, $m \geq 3$. Assume that the maximal complex subbundle ${\mathcal C}$ of $TM$ and the maximal quaternionic subbundle ${\mathcal Q}$ of $TM$ are
both invariant under the shape operator of $M$. If $JN \in
{\mathfrak J}N$, then one of the following statements holds:
\begin{enumerate}
\item [$(\tu_{A})$] {$M$ has exactly four distinct constant principal curvatures
\begin{equation*}
\alpha = 2\coth(2r),\ \beta = \coth(r),\ \lambda_1 = \tanh(r), \ \lambda_2 = 0,
\end{equation*}
and the corresponding principal curvature spaces are
\begin{equation*}
T_\alpha = TM \ominus {\mathcal C},\ T_\beta = {\mathcal C} \ominus {\mathcal Q},\ T_{\lambda_1} = E_{-1},\ T_{\lambda_2} = E_{+1}.
\end{equation*}
The principal curvature spaces $T_{\lambda_1}$ and $T_{\lambda_2}$
are complex (with respect to $J$) and totally complex (with respect
to ${\mathfrak J}$).}
\item [$(\ho_{A})$]  {$M$ has exactly three distinct constant principal curvatures
\begin{equation*}
\alpha = 2,\ \beta = 1,\ \lambda = 0
\end{equation*}
with corresponding principal curvature spaces
\begin{equation*}
T_\alpha = TM \ominus {\mathcal C},\ T_\beta = ({\mathcal C} \ominus {\mathcal Q}) \oplus E_{-1},\ T_\lambda = E_{+1}.
\end{equation*}}
Here, $E_{+1}$ and $E_{-1}$ are the eigenbundles of $\p \po |_{\mathcal Q}$ with respect to the eigenvaleus $+1$ and $-1$, respectively.
\end{enumerate}
\end{pro A}

Since the symmetric tensor $A\p-\p A$ vanishes identically on $\tu_{A}$ (resp. $\ho_{A}$), it trivially satisfies \eqref{eq: 3.12}. Hence we assert that  $\tu_{A}$ (resp., $\ho_{A}$) in complex hyperbolic two-plane Grassmannians $\NBt$ has the our commuting condition \eqref{C-1} (see~\cite{S5}).

\vskip 3pt

Next, due to Lemma~\ref{lemma 3.1}, let us suppose that $\x \in \Q$ (i.e., $JN  \perp {\mathfrak J}N$).

\noindent By virtue of the result in \cite{S7}, we assert that a Hopf hypersurface $M$ in complex hyperbolic two-plane Grassmannians $\NBt$ satisfying the hypotheses in Theorem~$1$ is locally congruent to
\begin{enumerate}
\item [$(\tu_{B})$] {$M$ is an open part of a tube around a totally geodesic quaternionic hyperbolic space ${\mathbb H}H^n$ in $SU_{2,2n}/S(U_2U_{2n})$, $m = 2n$,}
\item [$(\ho_{B})$] {$M$ is an open part of a horosphere in
$SU_{2,m}/S(U_2U_m)$ whose center at infinity is singular and of
type $JN \perp {\mathfrak J}N$, or}
\item [$(\mathcal E)$] {The normal bundle $\nu M$ of $M$ consists of singular tangent vectors of type $JX \perp {\mathfrak J}X$,}
\end{enumerate}
when $\x\in\Q$. Hereafter, the model spaces of $\tu_{B}$, $\ho_{B}$ or $\mathcal E$ is denoted by $M_{B}$. Let us check whether the shape operator~$A$ of model spaces of $M_{B}$ satisfy our conditions, conversely. In order to do this, let us introduce the following proposition given by Suh \cite{S7}.
\begin{pro B}\label{Proposition B}
Let $M$ be a connected hypersurface in $SU_{2,m}/S(U_2U_m)$, $m \geq 3$. Assume that the maximal complex
subbundle ${\mathcal C}$ of $TM$ and the maximal quaternionic subbundle
${\mathcal Q}$ of $TM$ are both invariant under the shape operator of
$M$. If $JN \perp {\mathfrak J}N$, then one of the following statements
holds:
\begin{enumerate}
\item [$(\tu_{B}$)]{$M$ has five (four for $r = \sqrt{2}{\tanh}^{-1}(1/\sqrt{3})$ in which case $\alpha = \lambda_2$)
distinct constant principal curvatures
\begin{equation*}
\begin{split}
\alpha & = \sqrt{2}\tanh(\sqrt{2}r),\ \beta =
\sqrt{2}\coth(\sqrt{2}r),\ \gamma = 0, \\
\lambda_1 &= \frac{1}{\sqrt{2}}\tanh(\frac{1}{\sqrt{2}}r),\ \lambda_2 =
\frac{1}{\sqrt{2}}\coth(\frac{1}{\sqrt{2}}r),
\end{split}
\end{equation*}
and the corresponding principal curvature spaces are
\begin{equation*}
T_\alpha = TM \ominus \mathcal {C},\ T_\beta = TM \ominus \mathcal {Q}, \ T_\gamma = J(TM \ominus \mathcal {Q}) = JT_\beta.
\end{equation*}
The principal curvature spaces $T_{\lambda_1}$ and $T_{\lambda_2}$
are invariant under ${\mathfrak J}$ and are mapped onto each other
by $J$. In particular, the quaternionic dimension of
$SU_{2,m}/S(U_2U_m)$ must be even.}
\item [($\ho_{B})$] {$M$ has exactly three distinct constant principal curvatures
\begin{equation*}
\alpha = \beta = \sqrt{2},\ \gamma = 0,\ \lambda =
\frac{1}{\sqrt{2}}
\end{equation*}
with corresponding principal curvature spaces
\begin{equation*}
T_\alpha = TM \ominus (\mathcal {C} \cap \mathcal {Q}),\ T_\gamma =
J(TM \ominus \mathcal {Q}),\ T_\lambda = \mathcal {C} \cap \mathcal {Q} \cap J\mathcal {Q}.
\end{equation*}}
\item [$(\mathcal E)$] {$M$ has at least four distinct principal curvatures,
three of which are given by
\begin{equation*}
\alpha = \beta = \sqrt{2},\ \gamma = 0,\ \lambda =
\frac{1}{\sqrt{2}}
\end{equation*}
with corresponding principal curvature spaces
\begin{equation*}
T_\alpha = TM \ominus (\mathcal {C} \cap \mathcal {Q}),\ T_\gamma =
J(TM \ominus {\mathcal Q}),\ T_\lambda \subset \mathcal {C} \cap \mathcal {Q} \cap J\mathcal {Q}.
\end{equation*}
If $\mu$ is another (possibly nonconstant) principal curvature function, then $JT_{\mu} \subset T_\lambda$
and ${\mathfrak J}T_{\mu} \subset T_\lambda$. Thus, the corresponding multiplicities are
\begin{equation*}
m(\alpha) = 4,\quad  m(\gamma) = 3, \quad m(\lambda), \quad m(\mu).
\end{equation*}}
\end{enumerate}
\end{pro B}

Let us assume that the structure Jacobi operator~$R_{\xi}$ of $M_{B}$ satisfies the property~\eqref{C-1}.
The tangent space of $M_{B}$ can be splitted into
\begin{equation*}
TM=T_{\al_1}\oplus T_{\al_2}\oplus T_{\al_3} \oplus T_{\al_4} \oplus T_{\al_5},
\end{equation*}
where $T_{\al_1}=[\x]$, $T_{\al_2}=\text{span}\{\xo,\xtw,\xh\}$, $T_{\al_3}=\text{span}\{\p\xo,\p\xtw,\p\xh\}$ and $T_{\al_4} \oplus T_{\al_5}$ is the orthogonal complement of $T_{\al_1}\oplus T_{\al_2}\oplus T_{\al_3}$ in $TM$.
Since $\x \in \Q$ and $\p \PN \xi = \phi^{2}\xi_{\nu}=-\xi_{\nu}$, we have $R_{\xi}(\phi \xt)= -2\phi_{2}\xi$. From this and $\al_3=0$ for all $M_{B}$, our commuting condition \eqref{C-1} becomes
\begin{equation*}
R_{\xi}\phi A \xt - A R_{\xi} \phi \xt  = -2\al_2\p\xt.
\end{equation*}
It implies that the eigenvalue $\al_2$ vanishes, since $\phi \xt$ is a unit tangent vector field. But in Proposition B,
for $\tu_{B}$ (resp. $\ho_{B}$ or $\mathcal E$) we see that the eigenvalue $\al_2=\beta = \sqrt{2}\coth(\sqrt{2}r)$ (resp. $\al_2=\al= \frac{1}{\sqrt{2}}$) is non-vanishing. This gives us a contradiction.
%
%
%
%

\vskip 15pt

\section{Proof of Theorem 2}\label{section 4}
\setcounter{equation}{0}
\renewcommand{\theequation}{4.\arabic{equation}}
\vspace{0.13in}

In this section, by using geometric quantities in \cite{LSW}, \cite{LSW2}, \cite{PaSW}, \cite{S7}, \cite{S8}, and \cite{SW}, we give a complete proof of Theorem~$2$.
To prove it, we assume that $M$ is a Hopf hypersurface in $\NBt$ with commuting structure Jacobi operator and Ricci tensor, that is,
\begin{equation}
(\RXP) SX = S (\RXP)X.
\tag{C-2}
\end{equation}


\noindent From the definition of the Ricci tensor and the fundamental formulas in \cite[Section~$2$]{SW}, the Ricci tensor $S$ of $M$ in $\NBt$ is given by
\begin{equation}\label{eq: 4.2}
\begin{split}
2 SX & = -(4m+7)X + 3{\E}(X){\xi} + 2hAX - 2A^2X \\
& \quad + {\SN}\{ 3{\EN}(X){\xn} - {\EN}({\xi}){\pn}{\phi}X + {\EN}({\p}X){\pn}{\xi} + {\eta}(X){\EN}({\xi}){\xi}_{\nu}\},
\end{split}
\end{equation}
where $h$ denotes the trace of the shape operator $A$.

\vskip 5pt
Using equations \eqref{C-2} and \eqref{eq: 4.2}, we prove that the Reeb vector field $\x$ of $M$ belongs to either the distribution $\Q$ or the distribution $\QP$.
\begin{lemma}\label{lemma 4.1}
Let $M$ be a Hopf hypersurface in $\NBt$, $m \geq 3$, with \eqref{C-2}. If the principal curvature $\alpha=g(A\x, \x)$ is constant along the direction of $\x$, then $\x$ belongs to either the distribution~$\Q$ or the distribution~$\QP$.
\end{lemma}

\begin{proof}
In order to prove this lemma, for some unit vectors $X_{0} \in \Q$, $\x_{1} \in \QP$, we put
\begin{equation*}
\x = \eta(X_{0})X_{0}+\eta(\x_{1})\x_{1},
\tag{*}
\end{equation*}
where $\eta(X_{0})\eta(\x_{1})\neq 0$ is the assumption we will disprove in this proof by contradiction.
\vskip 5pt
Let $\mathfrak U=\{p \in M\,|\, \alpha(p) \neq 0 \}$ be the open subset of $M$. From now on, we discuss our arguments on $\U$.

By virtue of Lemma~\ref{lemma 2.2}, $\x\al=0$ gives $A\X=\alpha \X$ and $A\xo=\al\xo$. From \eqref{eq: 4.2}, we have
\begin{equation}\label{eq: 4.3-1}
\left \{
\begin{aligned}
& S\p\X=\ka\p\X,\\
& S\X=(-2m-4+h\al-\al^2)\X+2\e(\X)\x,\\
& S\xo=(-2m-2+h\al-\al^2)\xo+2\eo(\x)\x,\\
& S\x=(-2m-2+h\al-\al^2)\x+2\eo(\x)\xo,
\end{aligned}
\right.
\end{equation}
where $\ka:=-2m-4+h\si-\si^2$ and $\sigma=\frac{\al^2-2\e^2({\X})}{\al}$ on $\mathfrak U$.\\

\noindent Put $X=\p\X$ into \eqref{C-2}, we have
\begin{equation}\label{eq: 4.5}
\ka \RX(\X)=S\RX(\X).
\end{equation}
\noindent Taking the inner product of \eqref{eq: 4.5} with $\x$ and using \eqref{eq: 3.5} and \eqref{eq: 4.3-1}, we have $-2\al^2\e^2(\xo)\e(\X)=0$.
It implies that $\U=\emptyset$. Thus it must be $p \in M-\U$. The set $M - \U=\text{Int}(M - \U) \cup \partial (M - \U)$, where $\text{Int}$ (resp., $\partial$)
 denotes the interior (resp., the boundary) of $M-\U$, we consider the following two cases:

\vskip 3pt

\begin{itemize}
\item {{\bf Case 1.} $p \in \text{Int}(M-\U)$}
\end{itemize}

\vskip 2pt
If $p \in \text{Int}(M-\U)$, then $\alpha = 0$. Our lemma was proved on $\text{Int}(M-\U)$ by the equation \eqref{eq: 2.10-1} and \eqref{*}.

\vskip 7pt

\begin{itemize}
\item {{\bf Case 2.} $p \in \partial (M-\U)$}
\end{itemize}

\vskip 2pt

Since $p \in \partial (M-\U)$, there exists a sequence of points $p_{n}\in\U$ such that $p_{n} \rightarrow p$ with $\al(p)=0$ and $\al(p_{n})\neq 0$. Such a sequence will have an infinite subsequence where $\eta(\xo)=0$ (in which case $\x \in \Q$ at $p$, by the continuity) or an infinite subsequence where $\eta(\X)=0$ (in which case $\x \in \Q^{\bot}$ at $p$). Accordingly, we get a complete proof of the Lemma.


\end{proof}

Now, we shall divide our consideration into two cases that $\x$ belongs to either the distribution $\Q$ or the distribution $\QP$, respectively.
Let us consider the case $\x \in \QP$. We may put $\x=\x_{1}\in \QP$ for the sake of convenience. Then, \eqref{eq: 4.2} is simplified:
\begin{equation}\label{eq: 4.9}
\begin{split}
2SX&=-(4m+7)X+7\e(X)\x+2\etw(X)\xtw\\
  &\quad+2\eh(X)\xth-\po\p X+2h AX-2A^2X.
\end{split}
\end{equation}

By replacing $X$ as $AX$ into \eqref{eq: 4.9} and using \eqref{eq: 3.9}, we obtain
\begin{equation}\label{eq: 4.10}
2SAX=-(4m+6)AX+6\al\e(X)\x+2h A^2X-2A^3X
\end{equation}
Applying the shape operator $A$ to \eqref{eq: 4.9} and using \eqref{eq: 3.10}, we get
\begin{equation}\label{eq: 4.11}
2ASX=-(4m+6)AX+6\al\e(X)\x+2h A^2X-2A^3X.
\end{equation}
\noindent From \eqref{eq: 4.10} and \eqref{eq: 4.11}, we see that the Ricci tensor $S$ commutes with the shape operator $A$, that is,
\begin{equation}\label{eq: 4.12}
SA=AS.
\end{equation}



On the other hand, the equations \eqref{eq: 3.7} and \eqref{eq: 4.9} give us
\begin{equation}\label{eq: 4.13}
\begin{split}
&2\eta_{3}(SX)\x_{2}-2\eta_{2}(SX)\x_{3}+\p_{1} SX-\p SX \\
& \quad \quad \,\, = (2m+4)\{2\eh(X)\xtw-2\etw(X)\xth+\p X-\po X\} \\
& \quad \quad := \text{Rem}(X).
\end{split}
\end{equation}
Taking the symmetric part of \eqref{eq: 4.13}, we obtain
\begin{equation}\label{eq: 4.13-1}
2\eta_{3}(X)S\x_{2}-2\eta_{2}(X)S\x_{3}+S\p_{1}X - S\p X = \text{Rem}(X).
\end{equation}

\begin{lemma}\label{lemma 4.2}
Let $M$ be a Hopf hypersurface in
$\NBt$ with \eqref{C-2}. If $\x \in \QP$, then $S\p=\p S$.
\end{lemma}

\begin{proof}
By virtue of equation~\eqref{eq: 4.13} and \eqref{eq: 4.13-1}, we obtain
the left and right sides of \eqref{C-2}, respectively, as follows:
\begin{equation*}\label{eq: 4.14}
\begin{split}
2\RXP SX & = -\p SX + 2 \al A\p SX - 2\eta_{3}(SX)\x_{2}+2\eta_{2}(SX)\x_{3} - \p_{1} SX \\
         & = -2 \p SX + 2 \al A\p SX - \text{Rem}(X), \\
\end{split}
\end{equation*}
and
\begin{equation*}
\begin{split}
2 S\RXP X  &= - S\p X + 2 \al S A\p  X - 2\eta_{3}(X)S\x_{2}+2\eta_{2}(X)S\x_{3} - S\p_{1} X \\
           &= -2 S \p X + 2 \al SA\p X - \text{Rem}(X).
\end{split}
\end{equation*}
That is,
\begin{equation}\label{eq: 4.14}
\RXP SX = - \p SX + \al A\p SX -\frac{1}{2}\text{Rem}(X)
\end{equation}
and
\begin{equation}
\begin{split}
S\RXP X=-S\p X+\al SA\p X-\frac{1}{2}\text{Rem}(X).
\end{split}
\end{equation}

\noindent From these two equations, the condition~\eqref{C-2} is equivalent to
\begin{equation}\label{eq: 4.16}
\begin{split}
(S\p -\p S)X & = \al (SA\p -A\p S)X \\
             & = \al A (S\p - \p S)X,
\end{split}
\end{equation}
by virtue of our assertion that the shape operator~$A$ commutes the Ricci tensor~ $S$ with each other given in \eqref{eq: 4.12}.
\vskip 3pt


\noindent Taking the symmetric part of \eqref{eq: 4.16}, we have
\begin{equation}\label{eq: 4.17}
(S\p -\p S)X=\al (S\p-\p S)AX
\end{equation}
for all tangent vector fields $X$ on $M$.
\vskip 3pt

\noindent From \eqref{eq: 4.16} and \eqref{eq: 4.17}, we know
\begin{equation}\label{eq: 4.17-1}
\al A(S\p-\p S)=\al(S\p-\p S)A.
\end{equation}
\vskip 3pt
Let $\mathfrak U=\{ p \in M \,|\, \alpha (p) \neq 0\}$ be an open subset of $M$.
Then \eqref{eq: 4.17-1} implies the shape operator $A$ and the symmetric tensor $S\p-\p S$ commute with each other on $\U$. Hence they are simultaneous diagonalizable, there exists a common orthonormal basis $\{E_1,E_2,...,E_{4m-1}\}$ such that the shape operator $A$ and the tensor $S\p-\p S$  both can be diagonalizable.
In other words,
\begin{eqnarray}
& AE_i=\lambda_i E_i,\label{eq: 4.18} \\
& (S\p-\p S)E_i=\beta_iE_i,\label{eq: 4.19}
\end{eqnarray}
where $\lambda_i$ and $\beta_i$ are scalars for all $i=1,2,...4m-1$.

\vskip 3pt

Combining equations in \eqref{eq: 4.2}, we get
\begin{equation}\label{eq: 4.20}
S\p X - \p SX = h A\p X - A^{2} \p X - h \p AX + \p A^{2}X.
\end{equation}
\indent Using \eqref{eq: 4.18}, \eqref{eq: 4.19} and \eqref{eq: 4.20}, we obtain
\begin{equation}\label{eq: 4.21}
(S\p-\p S)E_i  = h A\p E_i - A^{2} \p E_i - h \lambda_i\p E_i +  \lambda_i^{2}\p E_i.
\end{equation}

Taking the inner product with $E_i$ into \eqref{eq: 4.21}, we have
\begin{equation*}
\begin{split}
\beta_i g(E_i,E_i)= h\lambda_i g(\p E_i,E_i ) - \lambda^{2}_{i}g(\p E_i,E_i)-h\lambda_i g(\p E_i,E_i ) + \lambda^{2}_{i}g(\p E_i,E_i)= 0.
\end{split}
\end{equation*}
\noindent Since $g(E_i,E_i)=1$, we get $\beta_i=0$ for all $i=1,2,...,4m-1$. This is equivalent to $(S\p-\p S)E_i=0$ for all $i=1,2,...,4m-1$. It follows that $S\p X = \p SX$ for any tangent vector field $X$ on $\U$.
Next, if $p \in \text{Int}(M-\U)$, then we see that $\alpha(p)=0$. From this, the equation~\eqref{eq: 4.16} gives $(S\p - \p S)$ vanishes identically on $\text{Int}(M-\U)$.

\vskip 3pt

Finally, let us assume that $p \in \partial (M-\U)$, where $\partial (M-\U)$ is the boundary of~$M-\U$. Then there exists a subsequence $\{p_{n}\} \subset \U$ such that $p_{n}\rightarrow p$. Since $(S \phi - \phi S)X(p_{n})=0$ on the open subset~$\U$ in $M$, by the continuity we also get $(S \phi - \phi S)X(p)=0$.
\end{proof}

\vskip 3pt

%

By virtue of the result given by Suh in \cite{S8}, we assert that
{\it if $\x\in\QP$, then a Hopf hypersurface $M$ in $\NBt$ with \eqref{C-2} is locally congruent to one of the following hypersurfaces:
\begin{enumerate}
\item [$(\tu_{A})$] {a tube over a totally geodesic $\NBo$ in $\NBt$ or,}
\item [$(\ho_{A})$] {a horosphere in $\NBt$ whose center at infinity is singular and of type $JX \in {\mathfrak J}X$.}
\end{enumerate}}


%

\noindent Moreover, when $\x \in \QP$, \eqref{C-2} is equivalent to \eqref{eq: 4.16}. Since the symmetric tensor $(S\p-\p S)$ vanishes identically on $\tu_{A}$ (resp. $\ho_{A}$), it trivially satisfies \eqref{eq: 4.16}. Hence we assert that  $\tu_{A}$ (resp., $\ho_{A}$) in complex hyperbolic two-plane Grassmannians $\NBt$ has the our commuting condition \eqref{C-2} (see~\cite{S8}).
\vskip 3pt

When $\x \in \Q$, a Hopf hypersurface $M$ in $\NBt$ with \eqref{C-2} is locally congruent to a hypersurface of $M_{B}$ by \cite{S7}.
 From now on, let us show whether model spaces of $M_{B}$ satisfy the condition \eqref{C-2} or not. Then the tangent space of $M_{B}$ can be
splitted into
\begin{equation*}
TM_{B}=T_{\al_1}\oplus T_{\al_2}\oplus T_{\al_3} \oplus T_{\al_4} \oplus T_{\al_5}.
\end{equation*}
where $T_{\al_1}=[\x]$, $T_{\al_2}=\text{span}\{\xo,\xtw,\xh\}$, $T_{\al_3}=\text{span}\{\p\xo,\p\xtw,\p\xh\}$ and $T_{\al_4} \oplus T_{\al_5}$
is the orthogonal complement of $T_{\al_1}\oplus T_{\al_2}\oplus T_{\al_3}$ in $TM$ such that $JT_{\al_5}\subset T_{\al_4}$ (see \cite{S8}).

On $T_{p}M_{B}$, $p \in M_{B}$, the equations \eqref{eq: 4.2} and \eqref{eq: 3.2} are reduced to the following equations, respectively:
\begin{equation*}
\begin{split}
2SX&=-(4m+7)X + 3\eta(X)\x+ 2hAX - 2A^2 X\\
   &\quad\quad + \sum_{\nu =1}^3 \{ 3 \eta_{\nu}(X)\x_{\nu} +\eta(\p_{\nu}X)\p_{\nu}\x \},\\
2R_{\x}(X)&= -X + {\eta}(X){\x}+2{\alpha}AX-2\al^2{\eta}(X){\x}\\
          &\quad\quad+\SN\{{\EN}(X){\XN}+3{\EN}(\p X){\PNK}\}.
\end{split}
\end{equation*}

\indent From \cite[Proposition 5.1]{S8}, we obtain the following
\begin{equation}\label{eq: 4.22}
SX = \left\{ \begin{array}{ll}
                (-2m-2+h\al_1-\al_1^2)\x  & \mbox{if}\ \  X=\x \in T_{\al_1}\\
                (-2m-2+h\al_2-\al_2^2)\x_{\ell}  & \mbox{if}\ \  X=\x_{\ell} \in T_{\al_2} \\
                (-2m-4)\p\x_{\ell}  & \mbox{if}\ \ X=\p \x_{\ell} \in T_{\al_3}\\
                (-2m-\frac{7}{2}+h\al_4-\la_4^2) X  & \mbox{if}\ \  X \in T_{\al_4}\\
                (-2m-\frac{7}{2}+h\al_5-\al_5^2) X      & \mbox{if}\ \  X \in T_{\al_5}\\
\end{array}\right.
\end{equation}

\begin{equation}\label{eq: 4.23}
\RX (X) = \left\{ \begin{array}{ll}
                0                                       & \mbox{if}\ \  X=\x \in T_{\al_1}\\
                \al_1\al_2\x_{\ell} & \mbox{if}\ \  X=\x_{\ell} \in T_{\al_2} \\
                (-2+\al_1\al_3)\p\x_{\ell} & \mbox{if}\ \ X=\p \x_{\ell} \in T_{\al_3}\\
                (-\frac{1}{2}+\al_1\al_4) X & \mbox{if}\ \  X \in T_{\al_4}\\
                (-\frac{1}{2}+\al_1\al_5) X     & \mbox{if}\ \  X \in T_{\al_5}.
\end{array}\right.
\end{equation}
%
%



In order to check whether $\tu_{B}$, $\ho_{B}$ or $\mathcal E$ model spaces satisfy the \eqref{C-2} or not,  we should verify the following equations vanishes for all cases.

\begin{equation}\label{eq: 4.24}
G(X):= (\RXP)SX - S(\RXP)X.
\end{equation}

Putting $X=\xo\in T_{\al_3}$ into \eqref{eq: 4.24}, we have $G(\xo)=-2(2+\al_2h-{\al_2}^2)\p\xo$ which derives
\begin{equation}\label{eq: 4.25}
2+\al_2h-{\al_2}^2=0.
\end{equation}

\begin{itemize}
\item {{\bf Case 1.} Tube $\tu_{B}$}
\end{itemize}
In this case, we get $\al_1=\al$, $\al_2=\beta$, $\al_3=\gamma=0$, $\al_4=\lambda$ and $\al_5=\mu$.

\indent By calculation, we have $\lambda+\mu=\be$ on $\tu_{B}$. Thus we obtain $h=\al+3\beta+(4n-4)(\la+\mu) =\al+(2m-1)\beta$.
Then \eqref{eq: 4.25} is $4+2(m-1)\be^2>0$, which is a contradiction.

\begin{itemize}
\item {{\bf Case 2.} Horoshere $\ho_{B}$}
\end{itemize}
On $\ho_{B}$, $\al_1=\sqrt{2}$, $\al_2=\sqrt{2}$, $\al_3=\gamma=0$, $\al_4=\frac{1}{\sqrt{2}}$ and $\al_5=\frac{1}{\sqrt{2}}$.
Thus \eqref{eq: 4.25} gives $h=0$.
Since $h=\al_1+3\al_2+3\al_3+(4n-4)(\al_4+\al_5)$, we have $2\sqrt{2}m=0$ which is a contradiction.

\begin{itemize}
\item {{\bf Case 3.} Exceptional case $\mathcal E$}
\end{itemize}
For $X\in T_{\al_5}\subset T_{\mathcal E}$, $G(X)=-\frac{1}{2}(\al_5-\al_4)(\al_5+\al_4)\p X$. On $T_{\mathcal E}$ we have $\al_1=\alpha=\sqrt{2}$, $\al_4=\lambda=\frac{1}{\sqrt{2}}$ and $\al_5=\mu = \pm\frac{1}{\sqrt{2}}$. Because $\mu \neq \lambda$, it should be $\mu=-\frac{1}{\sqrt{2}}$. Moreover, since $J T_{\mu} \subset T_{\lambda}$ and $\mathfrak J T_{\mu} \subset T_{\lambda}$, we see that the corresponding multiplicities of the eigenvalues $\lambda$ and $\mu$ satisfy $m(\lambda) \geq m(\mu)$. Since $m(\al)=4$, $m(\gamma)=3$ and $m(\lambda)+ m(\mu)=4m-8$ on $\mathcal E$, the trace of the shape operator~$A$ denoted by $h$ becomes $h=4\alpha + 3\gamma + m(\lambda)\lambda + m(\mu) \mu = 4\sqrt{2} + \frac{1}{\sqrt{2}} (m(\lambda) - m(\mu))$, which makes a contradiction. In fact, since we obtained $h=0$ on $T_{\gamma} \in T \mathcal E$, it yields $(m(\lambda) - m(\mu))=-8 < 0$. Thus,
this case does not occur.

This shows that hypersurfaces of $\tu_{B}$, $\ho_{B}$ or $\mathcal E$ cannot satisfy the condition \eqref{C-2}, and therefore in the situation of Theorem~$2$, the case $X\in\Q$ cannot occur. This completes the proof of Theorem~$2$.
\vskip 20pt


\end{document}